\documentclass{article}
\usepackage{amssymb}
\usepackage{amsmath}
\usepackage{amsfonts}
\usepackage{bbding}
\usepackage{amsfonts}
\usepackage{graphicx}

\catcode`ö=\active
\defö{\"{o}}
\newtheorem{theorem}{Theorem}[section]

\newtheorem{definition}[theorem]{Definition}

\newtheorem{lemma}[theorem]{Lemma}

\newtheorem{remark}[theorem]{Remark}

\newenvironment{proof}[1][Proof]{\noindent\textbf{#1.} }{\ \rule{0.5em}{0.5em}}

\begin{document}

\begin{center}
\ \textbf{\large Existence and Behavior Results For a Nonlocal
Nonlinear Parabolic Equation With Variable Exponent}

\bigskip

U\u{g}ur Sert$^{\dag}$, Eylem \"{O}zt\"{u}rk \footnotetext{%
$^\dag$ Corresponding Author U. Sert: Faculty of Science,
Department of Mathematics, Hacettepe University, 06800, Beytepe,
Ankara, Turkey. e-mail: usert@hacettepe.edu.tr
\par
E. \"{O}zt\"{u}rk: Faculty of Science, Department of Mathematics, Hacettepe
University, 06800, Beytepe, Ankara, Turkey. e-mail: eozturk@hacettepe.edu.tr%
}
\end{center}

\noindent \textbf{Abstract}. In this paper, we study the
solvability of a Cauchy- Dirichlet problem for nonlinear parabolic
equation with non standard growths and nonlocal terms. We show the
existence of weak solutions of the considered problem under more
general conditions. In addition, we obtain some results on the
behavior of the solution when the problem is homogeneous.
\bigskip


\bigskip


\section{Introduction}

\noindent The present paper deals with the existence and behavior
of the solution (uniqueness in some sense) for nonlinear
degenerate parabolic Dirichlet-type boundary value problem whose
model example is the following:
\begin{equation*}
\left\{
\begin{array}{l}
\frac{\partial {u}}{\partial {t}}-\sum\limits_{i=1}^{n}D_{i}\left(
\left\vert u\right\vert ^{p_{0}-2}D_{i}u\right) +a\left(
x,t,u\right) +g\left( x,t\right) \left\Vert u\right\Vert
_{L^{p}\left( \Omega \right)
}^{s}\left( t\right) =h\left( x,t\right), \\
u\left( x,0\right) =0=u_{0}\left(x\right),\hspace{0.3cm}u\mid
_{\Gamma _{T}}=0
\end{array}%
\right. \label{MP} \tag{1.1}
\end{equation*}
where $\left( x,t\right) \in Q_{T}:=\Omega \times \left(
0,T\right),$ $T>0,$
$\Gamma_{T}:=\partial\Omega\times\left[0,T\right],$ $\Omega
\subset \mathbb{R}^{n}\left( n\geq 3\right) $ is a bounded domain
which has sufficiently smooth boundary (at least Lipschitz
boundary), $D_{i}\equiv\partial/\partial x_{i}$ and $p_{0}\geq 2,
p, s\geq 1$
and $a:\Omega \times \left( 0,T\right) \times \mathbb{R}\rightarrow \mathbb{R},$ $%
a\left( x,t,\tau \right) $ is a function with variable
nonlinearity in $\tau ,$ (for example, $a\left( x,t,\tau \right)=
a_{0}\left( x,t\right) \left\vert \tau \right\vert ^{\alpha \left(
x,t\right) -1}+a_{1}\left( x,t\right) $) and $g$ is a real valued
measurable function different from zero almost everywhere on
$Q_{T}.$
\par Recently nonlinear parabolic equations with nonlocal terms have been considerably
studied; e.g
see(\cite{AFA,AND,BEB1,CHi1,CHi2,CHi3,KDEN,WDEN,DUZ}). Here,
``Nonlinear nonlocal term" denotes a function dependence in space
domain $\Omega$. There are many physical processes that could be
expressed by nonlocal mathematical models and investigated by many
authors. For example, Galaktionov and Levine \cite{GAL} presented
a general view to critical Fujita exponents in nonlinear parabolic
problems with nonlocal nonlinearities, Pao \cite{PAO2} considered
a nonlocal model obtained from combustion theory. Degenerate
parabolic equations with a nonlocal term which appears in a
population dynamics model that interacts with each other by
chemical means, were studied in \cite{AND,WDEN,FUR}.
\par Equation in \eqref{MP} is nonlinear with respect to the solution and for the case $p_{0}=2$, the
equation appears in an ignition model for a compressible reactive
gas which is a nonlocal reaction-diffusion equation,
\cite{AND,BEB2}. In this case the existence, uniqueness and
blow-up of nonnegative solutions to the problem of form \eqref{MP}
have been studied in \cite{PAO1,ROU,SAP,SOU,WANG}. Models similar
to \eqref{MP} may also arise in theory of biological species to
describe the density of some biological species while nonlocal
term and absorption terms cooperate and communicate during the
diffusion.
\par The general form of the boundary-value problems including the equations of type \eqref{MP} which
is known as the equation of Newtonian filtration is
\begin{equation*}
u_{t}=\Delta \varphi\left(u\right)+f,
\end{equation*}
\eqref{MP} is a parabolic equation with implicit degeneracy which
is similar to the equation of Newtonian polytropic filtration
\cite{DUB,KAL,LAD,LIO} i.e.
\begin{equation*}
 u_{t}=\Delta\left(\left\vert u\right\vert^{m-1}u\right)+f,
\end{equation*} where $m>1$. This equation is
parabolic for $u$ different from $0$ and degenerates for $u = 0$.
Equation above with $m > 1$ describes the non-stationary flow of a
compressible Newtonian fluid in a porous medium under polytropic
conditions.
\par Over the past decade, there has been an increasing interest in the study of degenerate parabolic equations that involves
variable exponents \cite{ALK,ANT1,ANT2}. In this paper, we
investigate the parabolic equation whenever the additional term
$f$ has variable exponent nonlinearity together with nonlocal
term. If we reorganize the main part of the equation by taking the
derivative in the sum outside, we obtain
\begin{equation*}
 u_{t}=\Delta\left(\left\vert u\right\vert^{p_{0}-2}u\right)+F\left(x,t,u,\left\Vert u\right\Vert _{L^{p}\left(
\Omega \right)},h\right),
\end{equation*}
to the best of our knowledge, by now there has not been any
studies on the existence of solutions for the parabolic equations
of the type \eqref{MP} with $F$ is a function of nonlinear
non-local term $\left\Vert u\right\Vert _{L^{p}\left( \Omega
\right) }^{s}\left( t\right)$ besides a function of $\left\vert
u\right\vert$ with variable exponent nonlinearity. We stress that
since the nonlinearity of nonlocal term $g\left( x,t\right)
\left\Vert u\right\Vert _{L^{p}\left( \Omega \right) }^{s}\left(
t\right)$ is independent from the local nonlinearity, this causes
some difficulties in studying on uniqueness and on behavior of the
solution of the considered problem \eqref{MP}.
\par We use the general solvability theorem \cite{S4}, (Theorem \ref{gex}) to prove
the existence of weak solution of problem \eqref{MP}. We study the
posed problem in the space, that generated by this problem and
prove the existence of sufficiently smooth, in some sense,
solution of the problem under more general (weak) conditions.
Investigating most of boundary value problem on its own space
leads to obtain better results. Henceforth here considered problem
is investigated on its own space. Unlike linear boundary value
problems, the sets generated by nonlinear problems are subsets of
linear spaces, but not possessing the linear structure (see
\cite{BAN,SER,S1,S2,S3,S4} and references therein).
\par This paper is organized as follows: In the next section, we recall
some useful results on the generalized Orlicz-Lebesgue spaces and
results on nonlinear spaces (pn-spaces). In Section 3, we present
the assumptions, definition of the weak solution and prove the
existence of weak solution for problem \eqref{MP}. In Section 4,
we present a result on behavior of the solutions of \eqref{MP}
when the problem is homogeneous.

\section{Preliminaries}

\subsection{Generalized Lebesgue spaces}

In this subsection, some available facts from the theory of the
generalized Lebesgue spaces also called Orlicz-Lebesgue spaces
will be introduced. We present these facts without proofs which
can be found in \cite{ADA,DIE,KOV,RAD}.

Let $\Omega $ be a Lebesgue measurable subset of $%
\mathbb{R}
^{n}$ such that $\left\vert \Omega \right\vert >0$. (Throughout the paper,
we denote by $\left\vert \Omega \right\vert $ the Lebesgue measure of $%
\Omega $). Let $p\left( x,t\right) \geq 1$ be a measurable bounded function
defined on the cylinder $Q_{T}=\Omega \times \left( 0,T\right) $ i.e.%
\begin{equation*}  \label{2.1}
1\leq p^{-}\equiv \underset{Q_{T}}{ess}\inf \left\vert p\left( x,t\right)
\right\vert \leq \underset{Q_{T}}{ess}\sup \left\vert p\left( x,t\right)
\right\vert \equiv p^{+}<\infty .  \tag{2.1}
\end{equation*}
Then on the set of all functions on $Q_{T}$ define the functional $\sigma
_{p}$ and $\left\Vert .\right\Vert _{p}$ by%
\begin{equation*}
\sigma _{p}\left( u\right) \equiv \int\limits_{Q_{T}}\left\vert
u\right\vert ^{p\left( x,t\right) }dxdt.
\end{equation*}%
and%
\begin{equation*}
\left\Vert u\right\Vert _{L^{p\left( x,t\right) }\left( Q_{T}\right) }\equiv
\inf \left\{ \lambda >0|\text{ }\sigma _{p}\left( \frac{u}{\lambda }\right)
\leq 1\right\} .
\end{equation*}%
The Generalized Lebesgue space is defined as follows:%
\begin{equation*}
L^{p\left( x,t\right) }\left( Q_{T}\right) :=\left\{ u:u\text{ }\text{is a
measurable real-valued function in }Q_{T},\text{ }\sigma _{p}\left( u\right)
<\infty \right\} .
\end{equation*}%
The space $L^{p\left( x,t\right) }\left( Q_{T}\right) $ becomes a Banach
space under the norm $\left\Vert .\right\Vert _{L^{p\left( x,t\right)
}\left( \Omega \right) }$which is so-called Luxemburg norm.\medskip

We present the following results for these spaces
\cite{KOV,RAD,RAO}:

\begin{lemma}
\label{lm2.1} Let $0<\left\vert \Omega \right\vert <\infty ,$ and $p_{1},$ $%
p_{2}$ holds \eqref{2.1} then%
\begin{equation*}
L^{p_{1}\left( x,t\right) }\left( Q_{T}\right) \subset
L^{p_{2}\left( x,t\right) }\left( Q_{T}\right) \iff \text{
}p_{2}\left( x,t\right) \leq p_{1}\left( x,t\right) \text{ for a.e
}\left(x,t\right)\in Q_{T}
\end{equation*}
\end{lemma}

\begin{lemma}
The dual space of $L^{p\left( x,t\right) }\left( Q_{T}\right) $ is $%
L^{p^{\ast }\left( x,t\right) }\left( Q_{T}\right) $ if and only if $p\in
L^{\infty }\left( Q_{T}\right) $. The space $L^{p\left( x,t\right) }\left(
Q_{T}\right) $ is reflexive if and only if%
\begin{equation*}
1<p^{-}\leq p^{+}<\infty
\end{equation*}%
here $p^{\ast }\left( x,t\right) \equiv \frac{p\left( x,t\right) }{p\left(
x,t\right) -1}.$
\end{lemma}

For $u\in L^{p\left( x,t\right) }\left( Q_{T} \right) $ and $v\in
L^{q\left( x,t\right) }\left( Q_{T} \right) $ where $p,$ $q$
satisfies \eqref{2.1} and $\frac{1}{p\left( x,t\right) }+\frac{1}{q\left( x,t\right) }%
=1$ the following inequalities are hold:

\begin{equation*}
\int\limits_{Q_{T}}\left\vert uv\right\vert dxdt\leq 2\left\Vert
u\right\Vert _{L^{p\left( x,t\right) }\left( Q_{T}\right) }\left\Vert
v\right\Vert _{L^{q\left( x,t\right) }\left( Q_{T}\right) }\text{ }
\end{equation*}%
and for all $u\in L^{p\left( x,t\right) }\left( \Omega \right) ,$
\begin{equation*}
\min \{\left\Vert u\right\Vert _{L^{p\left( x,t\right) }\left( Q_{T}\right)
}^{p^{-}},\left\Vert u\right\Vert _{L^{p\left( x,t\right) }\left(
Q_{T}\right) }^{p^{+}}\}\leq \sigma _{p}\left( u\right) \leq \max
\{\left\Vert u\right\Vert _{L^{p\left( x,t\right) }\left( Q_{T}\right)
}^{p^{-}},\left\Vert u\right\Vert _{L^{p\left( x,t\right) }\left(
Q_{T}\right) }^{p^{+}}\}.
\end{equation*}

\subsection{On pn-spaces}

In this subsection, we introduce some function classes which are
complete metric spaces and directly connected to the considered
problem. Also we give some embedding results for these spaces
\cite{S1,S2,S3,S4} (see also references cited therein).

\begin{definition}
Let $\alpha \geq 0,$ $\beta \geq 1$, $\varrho =\left( \varrho
_{1,..,}\varrho _{n}\right) $ is multi-index, $m\in
\mathbb{Z}
^{+},$ $\Omega \subset
\mathbb{R}
^{n}\left( n\geq 1\right) $ is bounded domain with sufficiently smooth
boundary.%
\begin{equation*}
S_{m,\alpha ,\beta }\left( \Omega \right) \equiv \left\{ u\in L^{1}\left(
\Omega \right) \mid \left[ u\right] _{S_{m,\alpha ,\beta }\left( \Omega
\right) }^{\alpha +\beta }\equiv \sum_{0\leq \left\vert \varrho \right\vert
\leq m}\left( \int\limits_{\Omega }\left\vert u\right\vert ^{\alpha
}\left\vert D^{\varrho }u\right\vert ^{\beta }dx\right) <\infty \right\}
\end{equation*}%
in particularly,
\begin{equation*}
\mathring{S}_{1,\alpha ,\beta }\left( \Omega \right) \equiv \left\{ u\in
L^{1}\left( \Omega \right) \mid \left[ u\right] _{\mathring{S}_{1,\alpha
,\beta }\left( \Omega \right) }^{\alpha +\beta }\equiv \sum_{i=1}^{n}\left(
\int\limits_{\Omega }\left\vert u\right\vert ^{\alpha }\left\vert
D_{i}u\right\vert ^{\beta }dx\right) <\infty \right\} \cap \left\{ u\mid
_{\partial \Omega }\equiv 0\right\}
\end{equation*}%
and for $p\geq1$,
\begin{equation*}
L^{p}\left( 0,T;\mathring{S}_{1,\alpha ,\beta }\left( \Omega \right) \right)
\equiv \left\{ u\in L^{1}\left( Q_{T}\right) \mid \left[ u\right]
_{L^{p}\left( 0,T;\mathring{S}_{1,\alpha ,\beta }\left( \Omega \right)
\right) }^{p}\equiv \int\limits_{0}^{T}\left[ u\right] _{\mathring{S}%
_{1,\alpha ,\beta }\left( \Omega \right) }^{p}dt<\infty \right\}.
\end{equation*}%
These spaces are called pn-spaces.\footnote{$S_{1,\alpha ,\beta }\left(
\Omega \right) $ is a complete metric space with the following metric: $%
\forall u,v\in S_{1,\alpha ,\beta }\left( \Omega \right) $%
\par
\begin{equation*}
d_{S_{1,\alpha ,\beta }}\left( u,v\right) =\left\Vert \left\vert
u\right\vert ^{\frac{\alpha }{\beta }}u-\left\vert v\right\vert ^{\frac{%
\alpha }{\beta }}v\right\Vert _{W^{1,\beta }\left( \Omega \right) }
\end{equation*}%
}
\end{definition}

\begin{theorem}
\label{homeo} Let $\alpha \geq 0,$ $\beta \geq 1$ then $\varphi :%
\mathbb{R}
\longrightarrow
\mathbb{R}
$, $\varphi \left( t\right) \equiv \left\vert t\right\vert ^{\frac{\alpha }{%
\beta }}t$ is a homeomorphism between $S_{1,\alpha ,\beta }\left( \Omega
\right) $ and $W^{1,\beta }\left( \Omega \right) $.
\end{theorem}

\begin{theorem}
\label{embed} The following embeddings are satisfied:

\begin{itemize}
\item[(i)] Let $\alpha ,$ $\alpha _{1}\geq 0$ and $\beta _{1}\geq 1$, $\beta
\geq \beta _{1},$ $\frac{\alpha _{1}}{\beta _{1}}\geq \frac{\alpha }{\beta }%
, $ $\alpha _{1}+\beta _{1}\leq \alpha +\beta $ then we have%
\begin{equation*}
\mathring{S}_{1,\alpha ,\beta }\left( \Omega \right) \subseteq \mathring{S}%
_{1,\alpha _{1},\beta _{1}}\left( \Omega \right).
\end{equation*}

\item[(ii)] Let $\alpha \geq 0,$ $\beta \geq 1,$ $n>\beta $ and $\frac{%
n\left( \alpha +\beta \right) }{n-\beta }\geq r$ then there is a continuous
embedding%
\begin{equation*}
\mathring{S}_{1,\alpha ,\beta }\left( \Omega \right) \subset
L^{r}\left( \Omega \right).
\end{equation*}%
Furthermore for $\frac{n\left( \alpha +\beta \right) }{n-\beta }>r$ the
embedding is compact.

\item[(iii)] If $\alpha \geq 0,$ $\beta \geq 1$ and $p\geq \alpha +\beta \ $%
then%
\begin{equation*}
W_{0}^{1,p}\left( \Omega \right) \subset \mathring{S}_{1,\alpha ,\beta
}(\Omega )
\end{equation*}%
is hold.
\end{itemize}
\end{theorem}

\bigskip Now we present general solvability theorem \cite{S4}, the proof of it is based on Galerkin approximation (see also for
similar theorems \cite{S1,S3}), that will be used to prove the
existence of a weak solution of problem \eqref{MP}.

\begin{theorem}
\label{gex} Let $X$ and $Y$ be Banach spaces with dual spaces $X^{\ast }$
and $Y^{\ast }$ respectively, $Y$ be a reflexive Banach space, $\mathit{M}%
_{0}\subseteq X$ be a weakly complete \textquotedblleft
reflexive\textquotedblright\ pn-space, $X_{0}\subseteq \mathit{M}_{0}\cap Y$
be a separable vector topological space. Let the following conditions be
fulfilled:

\begin{itemize}
\item[(i)] $f:S_{0}\longrightarrow L^{q}\left( 0,T;Y\right) $ is a weakly
compact (weakly continuous) mapping, where%
\begin{equation*}
S_{0}:=L^{p}\left( 0,T;\mathit{M}_{0}\right) \cap W^{1,q}\left( 0,T;Y\right)
\cap \left\{ x\left( t\right) :x\left( 0\right) =0\right\}
\end{equation*}%
$1<\max \left\{ q,q^{\prime }\right\} \leq p<\infty ,$ $q^{\prime }=\frac{q}{%
q-1}$;

\item[(ii)] there is a linear continuous operator $A:W^{s,m}\left(
0,T;X_{0}\right) \longrightarrow W^{s,m}\left( 0,T;Y^{\ast }\right) ,$ $%
s\geq 0,$ $m\geq 1$ such that $A$ commutes with $\frac{\partial }{\partial t}
$ and the conjugate operator $A^{\ast }$ has ker$\left( A^{\ast }\right) =0$;

\item[(iii)] operators $f$ and $A$ generate, in generalized sense, a
coercive pair on space $L^{p}\left( 0,T;X_{0}\right) ,$ i.e. there exist a
number $r>0$ and a function $\Psi :%
\mathbb{R}
_{+}^{1}\longrightarrow
\mathbb{R}
_{+}^{1}$ such that $\Psi \left( \tau \right) /\tau \nearrow \infty $ as $%
\tau \nearrow \infty $ and for any $x\in L^{p}\left( 0,T;X_{0}\right) $
under $\left[ x\right] _{L^{p}\left( \mathit{M}_{0}\right) }\geq r$
following inequality holds:
\begin{equation*}
\int\limits_{0}^{T}\left\langle f\left( t,x\left( t\right) \right) ,Ax\left(
t\right) \right\rangle dt\geq \Psi \left( \left[ x\right] _{L^{p}\left(
\mathit{M}_{0}\right) }\right) ;
\end{equation*}

\item[(iv)] there exists some constants $C_{0}>0,$ $C_{1},C_{2}\geq 0$ and $%
\nu >1$ such that the inequalities%
\begin{align*}
& \int\limits_{0}^{T}\left\langle \xi \left( t\right) ,A\xi \left( t\right)
\right\rangle dt\geq C_{0}\left\Vert \xi \right\Vert _{L^{q}\left(
0,T;Y\right) }^{\nu }-C_{2}, \\
& \int\limits_{0}^{t}\left\langle \frac{\partial x}{\partial \tau },Ax\left(
\tau \right) \right\rangle d\tau \geq C_{1}\left\Vert x\right\Vert _{Y}^{\nu
}\left( t\right) -C_{2},\text{ \ \ a.e. }t\in \left[ 0,T\right]
\end{align*}%
hold for any $x\in W^{1,p}\left( 0,T;X_{0}\right) $ and $\xi \in L^{p}\left(
0,T;X_{0}\right) .$
\end{itemize}

Assume that that conditions (i)-(iv) are fulfilled. Then the Cauchy problem%
\begin{equation*}
\frac{dx}{d\tau }+f\left( t,x\left( t\right) \right) =y\left( t\right) ,%
\text{ \ \ \ }y\in L^{q}\left( 0,T;Y\right) ;\text{ \ }x\left( 0\right) =0
\end{equation*}%
is solvable in $S_{0}$ in the following sense%
\begin{equation*}
\int\limits_{0}^{T}\left\langle \frac{dx}{d\tau }+f\left( t,x\left( t\right)
\right) ,y^{\ast }\left( t\right) \right\rangle
dt=\int\limits_{0}^{T}\left\langle y\left( t\right) ,y^{\ast }\left(
t\right) \right\rangle ,\text{ \ \ }\forall y^{\ast }\in L^{q^{\prime
}}\left( 0,T;Y^{\ast }\right) ,
\end{equation*}%
for any $y\in L^{q}\left( 0,T;Y\right) $ satisfying the inequality%
\begin{equation*}
\sup \left\{ \frac{1}{\left[ x\right] _{L^{p}\left( 0,T;\mathit{M}%
_{0}\right) }}\int\limits_{0}^{T}\left\langle y\left( t\right) ,Ax\left(
t\right) \right\rangle dt:x\in L^{p}\left( 0,T;X_{0}\right) \right\} <\infty
.
\end{equation*}
\end{theorem}

\section{Statement of The Problem and The Main Result}

Let $\Omega \subset \mathbb{R}^{n}\left( n\geq 3\right) $ be a
bounded domain with sufficiently smooth boundary $\partial \Omega
.$ We study the problem
\begin{equation*}
\begin{cases}
\frac{\partial {u}}{\partial {t}}-\sum\limits_{i=1}^{n}D_{i}\left(
\left\vert u\right\vert ^{p_{0}-2}D_{i}u\right) +a\left( x,t,u\right)
+g\left( x,t\right) \left\Vert u\right\Vert _{L^{p}\left( \Omega \right)
}^{s}\left( t\right) =h\left( x,t\right) , & \left( x,t\right) \in Q_{T} \\
u\left( x,0\right) =0=u_{0}\left(x\right),\hspace{0.3cm}u\mid _{\Gamma _{T}}=0 &
\end{cases}%
\end{equation*}%
\eqref{MP} under the following conditions: \medskip \par
$p_{0}\geq 2,$ $p,s\geq 1,$ $g:Q_{T}\rightarrow \mathbb{R}$ is a
measurable function different from zero almost everywhere on $Q_{T}$ and $%
a:\Omega \times \left( 0,T\right) \times \mathbb{R}\rightarrow \mathbb{R},$ $%
a\left( x,t,\tau \right) $ is a Carath\'{e}dory  function with
variable nonlinearity in $\tau $ (see inequality \eqref{esit1}).
\medskip
\\Let the function $a\left( x,t,\tau \right) $ in problem
\eqref{MP} hold the following conditions: \medskip

\textbf{(U1)} \textit{There exists a measurable function} $\alpha
:\Omega \times \left( 0,T\right) \longrightarrow \mathbb{R}$,
$1<\alpha ^{-}\leq
\alpha \left( x,t\right) \leq \alpha ^{+}<\infty$ \textit{such that} $%
a\left( x,t,\tau \right) $\textit{\ satisfies the inequalities}%
\begin{equation}  \label{esit1}
\left\vert a\left( x,t,\tau \right) \right\vert \leq a_{0}\left( x,t\right)
\left\vert \tau \right\vert ^{\alpha \left( x,t\right) -1}+a_{1}\left(
x,t\right)  \tag{3.1}
\end{equation}
and
\begin{equation}  \label{esit2}
a\left( x,t,\tau \right) \tau \geq a_{2}\left( x,t\right) \left\vert \tau
\right\vert ^{\alpha \left( x,t\right) }-a_{3}\left( x,t\right),  \tag{3.2}
\end{equation}
a.e. $\left( x,t,\tau \right) \in Q_{T}\times \mathbb{R}.$

\textit{\ Here} $a_{i},$ $i=0,1,2,3$\textit{\ are nonnegative, measurable
functions defined on }$Q_{T}$ and $a_{2}\left( x,t\right) \geq A_{0}>0$%
\textit{\ a.e. }$\left( x,t\right) \in Q_{T}$. \bigskip \newline
We investigate problem \eqref{MP} for the functions $h\in
L^{q_{0}}\left( 0,T;W^{-1,q_{0}}\left( \Omega \right) \right)
+L^{\alpha ^{\ast }\left( x,t\right) }\left( Q_{T}\right) $ where
$\alpha ^{\ast }$ is conjugate of $\alpha $ i.e. $\alpha ^{\ast
}\left( x,t\right) :=\frac{\alpha
\left( x,t\right) }{\alpha \left( x,t\right) -1}$ and the dual space $%
W^{-1,q_{0}}\left( \Omega \right) :=\left( W_{0}^{1,p_{0}}\left(
\Omega \right) \right) ^{\ast },$ $q_{0}:=\frac{p_{0}}{p_{0}-1}$.

Let us denote $S_{0}$ by
\begin{equation*}
S_{0}:=L^{p_{0}}\left( 0,T;\mathring{S}_{1,\left( p_{0}-2\right)
q_{0},q_{0}}\left( \Omega \right) \right) \cap L^{\alpha \left( x,t\right)
}\left( Q_{T}\right) \cap W^{1,q_{0}}\left( 0,T;W^{-1,q_{0}}\left( \Omega
\right) \right) \cap \{u:u\left( x,0\right) =0\}.
\end{equation*}
We understand the solution of the considered problem in the
following sense:
\begin{definition}
\label{weakdef} A function $u\in S_{0}$, is called the generalized
solution
(weak solution) of problem \eqref{MP} if it satisfies the equality%
\begin{align*}
& \int\limits_{0}^{T}\int\limits_{\Omega }\frac{\partial {u}}{\partial {t}}%
wdxdt+\sum_{i=1}^{n}\int\limits_{0}^{T}\int\limits_{\Omega }\left(
\left\vert u\right\vert ^{p_{0}-2}D_{i}u\right) D_{i}wdxdt \\
& +\int\limits_{0}^{T}\int\limits_{\Omega }a\left( x,t,u\right)
wdxdt+\int\limits_{0}^{T}\int\limits_{\Omega }g\left( x,t\right) \left\Vert
u\right\Vert _{L^{p}\left( \Omega \right)
}^{s}wdxdt=\int\limits_{0}^{T}\int\limits_{\Omega }hwdxdt
\end{align*}%
for all $w\in L^{p_{0}}\left( 0,T;W_{0}^{1,p_{0}}\left( \Omega \right)
\right) \cap L^{\alpha \left( x,t\right) }\left( Q_{T}\right) \cap
W^{1,q_{0}}\left( 0,T;W^{-1,q_{0}}\left( \Omega \right) \right) $.
\end{definition}

We now ready to proceed the main theorem of this section but first
denote the followings. For sufficiently small $\eta \in \left(
0,1\right)$
\begin{equation*}
Q_{1,T}:=\left\{ \left( x,t\right) \in Q_{T}|\text{ }\alpha \left(
x,t\right) \in \lbrack 1,p_{0}-\eta )\right\} ,
\end{equation*}%
\begin{equation*}
Q_{2,T}:=\left\{ \left( x,t\right) \in Q_{T}|\text{ }\alpha \left(
x,t\right) \in \lbrack p_{0}-\eta ,\alpha ^{+}]\right\}
\end{equation*}

and
\begin{equation*}
\beta \left( x,t\right) :=\left\{
\begin{array}{l}
\frac{p_{0}\alpha ^{\ast }\left( x,t\right) }{p_{0}-\alpha \left( x,t\right)
}\text{ if\ }\left( x,t\right) \in Q_{1,T} \\
\infty \text{ \ \ \ \ \ \ \ \ if\ }\left( x,t\right) \in Q_{2,T}%
\end{array}%
\right.
\end{equation*}

Also, $\tilde{p_{0}}:=\frac{np_{0}}{n-q_{0}}$ which is critical
exponent in Theorem \ref{embed} and its conjugate is
$\tilde{p_{0}}^{\ast }=\frac{\tilde{p_{0}}}{\tilde{p_{0}}-1}.$

\begin{theorem}\textbf{(Existence Theorem)}
\label{var} Let \textbf{(U1)} is hold; $1\leq s<p_{0}-1$ and
$p\leq p_{0}.$ If $a_{0}\in L^{\beta \left( x,t\right) }\left(
Q_{T}\right) ,$ $a_{1}\in L^{\alpha ^{\ast }\left( x,t\right)
}\left( Q_{T}\right) ,$ $a_{2}\in L^{\infty
}\left( Q_{T}\right) ,$ $a_{3}\in L^{1}\left( Q_{T}\right) $ and $g\in L^{%
\frac{p_{0}}{p_{0}-(s+1)}}\left( 0,T;L^{\tilde{p_{0}}^{\ast
}}\left( \Omega \right) \right) $ then for all $h\in
L^{q_{0}}\left( 0,T;W^{-1,q_{0}}\left( \Omega \right) \right)
+L^{\alpha ^{\ast }\left( x,t\right) }\left(
Q_{T}\right) $ problem \eqref{MP} has a generalized solution in the space $%
S_{0}$ and $\partial u/\partial t$ belongs to $L^{q_{0}}\left(
0,T;W^{-1,q_{0}}\left( \Omega \right) \right).$
\end{theorem}

\subsection{The Proof of Theorem \ref{var}}
The proof is based on the general existence theorem (Theorem
\ref{gex}). For this, we introduce the following spaces and
mappings in order to apply Theorem \ref{gex} to prove Theorem
\ref{var}.
\begin{align*}
&S_{0}:=L^{p_{0}}\left( 0,T;\mathring{S}_{1,\left( p_{0}-2\right)
q_{0},q_{0}}\left( \Omega \right) \right) \cap L^{\alpha \left(
x,t\right) }\left( Q_{T}\right) \cap W^{1,q_{0}}\left(
0,T;W^{-1,q_{0}}\left( \Omega \right) \right) \cap \{u:u\left(
x,0\right) =0\},\\
&f:S_{0}\longrightarrow L^{q_{0}}\left( 0,T;W^{-1,q_{0}}\left(
\Omega \right) \right) +L^{\alpha ^{\ast }\left( x,t\right)
}\left( Q_{T}\right),\\
&f\left( u\right) :=-\sum_{i=1}^{n}D_{i}\left( \left\vert
u\right\vert ^{p_{0}-2}D_{i}u\right) +a\left( x,t,u\right)
+g\left( x,t\right) \left\Vert u\right\Vert _{L^{p}\left( \Omega
\right) }^{s}\left( t\right),\\
&A:L^{p_{0}}\left( 0,T;W_{0}^{1,p_{0}}\left( \Omega \right)
\right) \cap L^{\alpha \left( x,t\right) }\left( Q_{T}\right)
\subset S_{0}\longrightarrow L^{p_{0}}\left(
0,T;W_{0}^{1,p_{0}}\left( \Omega \right) \right) \cap L^{\alpha
\left( x,t\right) }\left( Q_{T}\right),\\
&A:=Id.
\end{align*}
\medskip We prove some lemmas to show that all conditions of Theorem \ref{gex} are fulfilled under the conditions of Theorem \ref{var}.

\begin{lemma}
\label{lemcoercive}Under the conditions of Theorem \ref{var}, $f$
and $A$ generate a ``coercive pair'' on $L^{p_{0}}\left(0,T;W^{1,p_{0}}_{0}\left(%
\Omega\right)\right) \cap L^{\alpha \left( x,t\right) }\left(
Q_{T} \right)$.
\end{lemma}

\begin{proof}
Since $A\equiv Id,$ being ``coercive pair'' equals to order
coercivity of $f$ on the space $L^{p_{0}}\left(
0,T;W_{0}^{1,p_{0}}\left( \Omega \right) \right) \cap L^{\alpha
\left( x,t\right) }\left( Q_{T}\right) $. For $u\in
L^{p_{0}}\left( 0,T;W_{0}^{1,p_{0}}\left( \Omega \right) \right)
\cap L^{\alpha \left( x,t\right) }\left( Q_{T}\right) $ we have,
\begin{align*}
\left\langle f\left( u\right) ,u\right\rangle
_{Q_{T}}&=\sum_{i=1}^{n}\left(
\int\limits_{0}^{T}\int\limits_{\Omega }\left\vert u\right\vert
^{p_{0}-2}\left\vert D_{i}u\right\vert
^{2}dxdt\right) \\
& +\int\limits_{Q_{T}}a\left( x,t,u\right)
udxdt+\int\limits_{0}^{T}\int\limits_{\Omega }g\left( x,t\right)
\left\Vert u\right\Vert _{L^{p}\left( \Omega \right) }^{s}udxdt.
\end{align*}
Using \eqref{esit2}, we obtain
\begin{align*}  \label{3.3}
\left\langle f\left( u\right) ,u\right\rangle _{Q_{T}}&\geq
\sum_{i=1}^{n}\left( \int\limits_{0}^{T}\int\limits_{\Omega
}\left\vert u\right\vert ^{p_{0}-2}\left\vert D_{i}u\right\vert
^{2}dxdt\right)+\int\limits_{Q_{T}}\left\vert a_{2}\left(
x,t\right)
\right\vert \left\vert u\right\vert ^{\alpha \left( x,t\right) }dxdt \\
&-\int\limits_{Q_{T}}\left\vert a_{3}\left( x,t\right) \right\vert
dxdt-\int\limits_{0}^{T}\int\limits_{\Omega }\left\vert g\left(
x,t\right) \right\vert \left\Vert u\right\Vert _{L^{p}\left(
\Omega \right) }^{s}\left\vert u\right\vert dxdt.  \tag{3.3}
\end{align*}
If we use \textbf{(U1)} to estimate the second integral in
\eqref{3.3} and H\"{o}lder inequality with the embedding
$\mathring{S}_{1,\left( p_{0}-2\right) q_{0},q_{0}}\left( \Omega
\right) \subset L^{p}\left( \Omega \right) $ in Theorem
\ref{embed} to estimate the fourth integral then we get,
\begin{align*}  \label{3.4}
\left\langle f\left( u\right) ,u\right\rangle _{Q_{T}}&\geq \left[
u\right] _{L^{p_{0}}\left( 0,T;\mathring{S}_{1,\left(
p_{0}-2\right) ,2}\left( \Omega \right) \right)
}^{p_{0}}+{A}_{0}\int\limits_{Q_{T}}\left\vert u\right\vert
^{^{\alpha \left( x,t\right) }}dxdt \\
& -C\int\limits_{0}^{T}\left[ u\right] _{\mathring{S}_{1,\left(
p_{0}-2\right) q_{0},q_{0}}\left( \Omega \right) }^{s}\left\Vert
u\right\Vert _{L^{\tilde{p_{0}}}\left( \Omega \right) }\left\Vert
g\right\Vert _{L^{\tilde{p_{0}}^{\ast }}\left( \Omega \right)
}dt\\
&-\left\Vert a_{3}\right\Vert _{L^{1}\left( Q_{T}\right) } .
\tag{3.4}
\end{align*}
By taking account the embeddings (see Theorem \ref{embed})
\begin{equation*}
\mathring{S}_{1,\left( p_{0}-2\right) ,2}\left( \Omega \right) \subset
\mathring{S}_{1,\left( p_{0}-2\right) q_{0},q_{0}}\left( \Omega \right)
\end{equation*}%
and
\begin{equation*}
\mathring{S}_{1,\left( p_{0}-2\right) q_{0},q_{0}}\left( \Omega
\right) \subset L^{\tilde{p_{0}}}\left( \Omega \right)
\end{equation*}%
into \eqref{3.4} to estimate the pseudo-norm and third integral
respectively, we obtain
\begin{align*}  \label{3.5}
\left\langle f\left( u\right) ,u\right\rangle _{Q_{T}}&\geq C_{0}\left[ u%
\right] _{L^{p_{0}}\left( 0,T;\mathring{S}_{1,\left( p_{0}-2\right)
q_{0},q_{0}}\left( \Omega \right) \right) }^{p_{0}}+{A}_{0}\int%
\limits_{Q_{T}}\left\vert u\right\vert ^{^{\alpha \left( x,t\right) }}dxdt \\
& -C_{1}\int\limits_{0}^{T}\left[ u\right]
_{\mathring{S}_{1,\left( p_{0}-2\right) q_{0},q_{0}}\left( \Omega
\right) }^{s+1}\left\Vert g\right\Vert _{L^{\tilde{p_{0}}^{\ast
}}\left( \Omega \right) }dt-\left\Vert a_{3}\right\Vert
_{L^{1}\left( Q_{T}\right) } .  \tag{3.5}
\end{align*}%
By using Young's inequality to the third integral in \eqref{3.5}%
, we have
\begin{equation*}
\left\langle f\left( u\right) ,u\right\rangle _{Q_{T}}\geq C_{2}\left( \left[
u\right] _{L^{p_{0}}\left( 0,T;\mathring{S}_{1,\left( p_{0}-2\right)
q_{0},q_{0}}\left( \Omega \right) \right) }^{p_{0}}+\left\Vert u\right\Vert
_{L^{\alpha \left( x,t\right) }\left( Q_{T}\right) }^{\alpha ^{-}}\right) -K.
\end{equation*}%
Here, $K=K\left( \left\Vert a_{3}\right\Vert _{L^{1}\left( Q_{T}\right)
},\left\Vert g\right\Vert _{L^{\frac{p_{0}}{p_{0}-\left( s+1\right) }}\left(
0,T;L^{\tilde{p_{0}}^{\ast }}\left( \Omega \right) \right) }\right) $, $%
C_{2}=C_{2}\left( p_{0},s,{A}_{0},\left\vert \Omega \right\vert
\right) $ are positive constants. So the proof is completed.
\end{proof}

\begin{lemma}
\label{lembounded}Under the conditions of Theorem \ref{var}, $f$
is bounded from $S_{0}$ into
$L^{q_{0}}\left(0,T;W^{-1,q_{0}}\left( \Omega \right)\right)
+L^{\alpha ^{\ast }\left( x,t\right) }\left( Q_{T} \right)$.
\end{lemma}

\begin{proof}
Firstly we define the mappings%
\begin{align*}
& f_{1}\left( u\right) :=\sum_{i=1}^{n}-D_{i}\left( \left\vert u\right\vert
^{p_{0}-2}D_{i}u\right) +g\left( x,t\right) \left\Vert u\right\Vert
_{L^{p}\left( \Omega \right) }^{s}\left( t\right), \\
& f_{2}\left( u\right) :=a\left( x,t,u\right) .
\end{align*}%
We need to show that, these mappings are both bounded from $L^{p_{0}}\left( 0,T;%
\mathring{S}_{1,\left( p_{0}-2\right) q_{0},q_{0}}\left( \Omega \right)
\right) \cap L^{\alpha \left( x,t\right) }\left( Q_{T}\right) $ to $%
L^{q_{0}}\left( 0,T;W^{-1,q_{0}}\left( \Omega \right) \right) +L^{\alpha
^{\ast }\left( x,t\right) }\left( Q_{T}\right) .$

Let us show that $f_{1}$ is bounded: For $u\in L^{p_{0}}\left( 0,T;\mathring{%
S}_{1,\left( p_{0}-2\right) q_{0},q_{0}}\left( \Omega \right) \right) $ and $%
v\in L^{p_{0}}\left( 0,T;W_{0}^{1,p_{0}}\left( \Omega \right) \right) $%
\begin{equation*}
\left\vert \left\langle f_{1}\left( u\right) ,v\right\rangle
_{Q_{T}}\right\vert \leq \sum_{i=1}^{n}\left(
\int\limits_{0}^{T}\int\limits_{\Omega }\left\vert u\right\vert
^{p_{0}-2}\left\vert D_{i}u\right\vert \left\vert
D_{i}v\right\vert dxdt\right)
+\int\limits_{0}^{T}\int\limits_{\Omega }\left\vert g\left(
x,t\right) \right\vert \left\Vert u\right\Vert _{L^{p}\left(
\Omega \right) }^{s}\left\vert v\right\vert dxdt.
\end{equation*}%
Using the embedding $\mathring{S}_{1,\left( p_{0}-2\right)
q_{0},q_{0}}\left( \Omega \right) \subset L^{p}\left( \Omega \right) $ and H%
\"{o}lder's inequality above we get,
\begin{align*}
& \leq \left[ \sum_{i=1}^{n}\left( \int\limits_{0}^{T}\int\limits_{\Omega
}\left\vert u\right\vert ^{(p_{0}-2)q_{0}}\left\vert D_{i}u\right\vert
^{q_{0}}dxdt\right) \right] ^{\frac{1}{q_{0}}}\left[ \sum_{i=1}^{n}\left(
\int\limits_{0}^{T}\int\limits_{\Omega }\left\vert D_{i}v\right\vert
^{p_{0}}dxdt\right) \right] ^{\frac{1}{p_{0}}} \\
& +\tilde{C}\int\limits_{0}^{T}\left[ u\right] _{\mathring{S}_{1,\left(
p_{0}-2\right) q_{0},q_{0}}\left( \Omega \right) }^{s}\left\Vert
g\right\Vert _{L^{\frac{np_{0}}{n(p_{0}-1)+p_{0}}}\left( \Omega \right)
}\left\Vert v\right\Vert _{W_{0}^{1,p_{0}}\left( \Omega \right) }dt
\end{align*}
Estimating the second integral above by H\"{o}lder's inequality ($\frac{p_{0}%
}{s}>1$), we obtain
\begin{align*}
& \left\vert \left\langle f_{1}\left( u\right) ,v\right\rangle
_{Q_{T}}\right\vert\leq \Psi(\left[ u\right] _{L^{p_{0}}\left(
0,T;\mathring{S}_{1,\left( p_{0}-2\right) q_{0},q_{0}}\left(
\Omega \right) \right) })\left\Vert v\right\Vert _{L^{p_{0}}\left(
0,T;W_{0}^{1,p_{0}}\left( \Omega \right) \right) }
\end{align*}%
where
\begin{align*}&\Psi(\left[ u\right] _{L^{p_{0}}\left(
0,T;\mathring{S}_{1,\left( p_{0}-2\right) q_{0},q_{0}}\left(
\Omega \right) \right) })=\\
&\left[ u\right] _{L^{p_{0}}\left( 0,T;\mathring{S}%
_{1,\left( p_{0}-2\right) q_{0},q_{0}}\left( \Omega \right)
\right) }^{p_{0}-1}+\tilde{C}_{1}\left[ u\right] _{L^{p_{0}}\left(
0,T;\mathring{S}_{1,\left( p_{0}-2\right) q_{0},q_{0}}\left(
\Omega \right) \right) }^{s}\left\Vert g\right\Vert
_{L^{\frac{p_{0}}{p_{0}-\left(s+1\right)}}\left(
0,T;L^{\tilde{p_{0}}^{\ast }}\left( \Omega \right)
\right)}.
\end{align*}
Thus by the last inequality we obtain the boundedness of $f_{1}$.
\par Similarly by using \eqref{esit1} and Theorem \ref{embed},
for all $u\in
S_{0} $, we have the following estimate%
\begin{align*}
\sigma _{\alpha ^{\ast }}\left( f_{2}\left( u\right) \right)
&=\sigma
_{\alpha ^{\ast }}\left( a\left( x,t,u\right) \right) \\
& =\int\limits_{0}^{T}\int\limits_{\Omega }\left\vert a\left(
x,t,u\right) \right\vert ^{\alpha ^{\ast }(x,t)}dxdt\\
&\leq C_{3}\left( \sigma _{\alpha
}\left( u\right) +\left[ u\right] _{L^{p_{0}}\left( 0,T;\mathring{S}%
_{1,\left( p_{0}-2\right) q_{0},q_{0}}\left( \Omega \right) \right)
}^{p_{0}}\right) +C_{4},
\end{align*}%
here $C_{3}=C_{3}\left( \alpha ^{+},\alpha ^{-},\left\Vert a_{0}\right\Vert
_{L^{\beta \left( x,t\right) }\left( Q_{T}\right) }\right),$ $%
C_{4}=C_{4}\left( \sigma _{\beta }\left( a_{0}\right) ,\sigma
_{\alpha ^{\ast }}\left( a_{1}\right) ,\left\vert \Omega
\right\vert \right)>0$ are
constants. So we prove that $f_{2}:$ $L^{p_{0}}\left( 0,T;\mathring{S}%
_{1,\left( p_{0}-2\right) q_{0},q_{0}}\left( \Omega \right) \right) \cap
L^{\alpha \left( x,t\right) }\left( Q_{T}\right) $ $\rightarrow $ $L^{\alpha
^{\ast }(x,t)}\left( Q_{T}\right) $ is bounded.
\end{proof}

\begin{lemma}
\label{lemweak}Under the conditions of Theorem \ref{var}, $f$ is
weakly compact from $S_{0}$ into $L^{q_{0}}\left(
0,T;W^{-1,q_{0}}\left( \Omega \right) \right) +L^{\alpha ^{\ast
}\left( x,t\right) }\left( Q_{T}\right) $.
\end{lemma}

\begin{proof}
First we verify the weak compactness of $f_{0}$ where $f_{0}\left( u\right)
:=-\sum_{i=1}^{n}D_{i}\left( \left\vert u\right\vert ^{p_{0}-2}D_{i}u\right)
$. Let $\left\{ u_{m}\left( x,t\right) \right\} _{m=1}^{\infty }\subset
S_{0} $ be bounded and $u_{m}\overset{\text{ }S_{0}}{\rightharpoonup }u_{0}$
it is sufficient to show a subsequence of $\left\{ u_{m_{j}}\right\}
_{m=1}^{\infty }\subset \left\{ u_{m}\right\} _{m=1}^{\infty }$ which
satisfies $f_{0}\left( u_{m_{j}}\right) $ $\overset{\text{ }L^{q_{0}}\left(
0,T;W^{-1,q_{0}}\left( \Omega \right) \right) }{\rightharpoonup }f_{0}\left(
u_{0}\right) .$

Since for fixed a.e. $t\in \left( 0,T\right) ,$ $u_{m}\left(
x,t\right)\in \mathring{S}_{1,\left( p_{0}-2\right)
q_{0},q_{0}}\left( \Omega \right) $ and we have one-to-one
correspondence
between the classes (Theorem \ref{homeo})%
\begin{equation*}
\mathring{S}_{1,\left( p_{0}-2\right) q_{0},q_{0}}\left( \Omega \right)
\underset{\varphi ^{-1}}{\overset{\varphi }{\longleftrightarrow }}%
W_{0}^{1,q_{0}}\left( \Omega \right)
\end{equation*}%
with the homeomorphism%
\begin{equation*}
\varphi \left( \tau \right) \equiv \left\vert \tau \right\vert
^{p_{0}-2}\tau ,\text{ }\varphi ^{-1}\left( \tau \right) \equiv \left\vert
\tau \right\vert ^{-\frac{p_{0}-2}{p_{0}-1}}\tau
\end{equation*}%
for $\forall m\geq 1$%
\begin{equation*}
\left\vert u_{m}\right\vert ^{p_{0}-2}u_{m}\in L^{q_{0}}\left(
0,T;W_{0}^{1,q_{0}}\left( \Omega \right) \right)
\end{equation*}%
is bounded. Since $L^{q_{0}}\left( 0,T;W_{0}^{1,q_{0}}\left(
\Omega \right) \right) $ is a reflexive space, there exists a
subsequence $\left\{ u_{m_{j}}\right\} _{m=1}^{\infty }\subset
\left\{ u_{m}\right\} _{m=1}^{\infty }$ such that
\begin{equation*}
\left\vert u_{m_{j}}\right\vert ^{p_{0}-2}u_{m_{j}}\overset{\text{ }%
L^{q_{0}}\left( 0,T;W_{0}^{1,q_{0}}\left( \Omega \right) \right) }{%
\rightharpoonup }\xi.
\end{equation*}%
Now we show that $\xi =\left\vert u_{0}\right\vert
^{p_{0}-2}u_{0}.$ According to compact embedding \cite{S1},
\begin{equation*}  \label{3.6}
L^{p_{0}}\left( 0,T;\mathring{S}_{1,\left( p_{0}-2\right)
q_{0},q_{0}}\left( \Omega \right) \right) \cap W^{1,q_{0}}\left(
0,T;W^{-1,q_{0}}\left( \Omega \right) \right) \hookrightarrow
L^{p_{0}}\left( Q_{T}\right)  \tag{3.6}
\end{equation*}
\begin{equation*}
\exists \left\{ u_{m_{j_{k}}}\right\} _{m=1}^{\infty }\subset \left\{
u_{m_{j}}\right\} _{m=1}^{\infty }\text{, }u_{m_{j_{k}}}\overset{\text{ }%
L^{p_{0}}\left( Q_{T}\right) }{\rightarrow }u_{0}\text{ }
\end{equation*}%
which implies%
\begin{equation*}
u_{m_{j_{k}}}\underset{a.e}{\overset{Q_{T}}{\rightarrow }}u_{0}
\end{equation*}%
by the continuity of $\varphi \left( \tau \right),$ we get
\begin{equation*}
\left\vert u_{m_{j_{k}}}\right\vert ^{p_{0}-2}u_{m_{j_{k}}}\underset{a.e}{%
\overset{Q_{T}}{\rightarrow }}\left\vert u_{0}\right\vert
^{p_{0}-2}u_{0}.
\end{equation*}%
So we obtain, $\xi =\left\vert u_{0}\right\vert ^{p_{0}-2}u_{0}.$

From this, we conclude that for $\forall v\in L^{p_{0}}\left(
0,T;W_{0}^{1,p_{0}}\left( \Omega \right) \right) $
\begin{align*}
\langle f_{0}\left( u_{m_{j_{k}}}\right) ,v\rangle_{Q_{T}}&
=\sum_{i=1}^{n}\langle -D_{i}\left(\left\vert u_{m_{j_{k}}}\right\vert
^{p_{0}-2}D_{i}u_{m_{j_{k}}}\right),v\rangle_{Q_{T}} \\
& \underset{m_{j}\nearrow \infty }{\longrightarrow }\sum_{i=1}^{n}\langle
-D_{i}\left( \left\vert u_{0}\right\vert ^{p_{0}-2}D_{i}u_{0}\right)
,v\rangle_{Q_{T}}=\langle f_{0}\left( u_{0}\right) ,v\rangle _{Q_{T}}
\end{align*}%
hence, the result is obtained.

Now we shall show the weak compactness of $f_{2}$. Since%
\begin{equation*}
a:L^{p_{0}}\left( 0,T;\mathring{S}_{1,\left( p_{0}-2\right)
q_{0},q_{0}}\left( \Omega \right) \right) \cap L^{\alpha \left( x,t\right)
}\left( Q_{T}\right) \rightarrow L^{\alpha ^{\ast }(x,t)}\left( Q_{T}\right)
\end{equation*}%
is bounded by Lemma \ref{lembounded}, then for $m\geq 1,$
$f_{2}\left( u_{m}\right) =\left\{ a\left( x,t,u_{m}\right)
\right\} _{m=1}^{\infty }\subset L^{\alpha ^{\ast }(x,t)}\left(
Q_{T}\right) .$ Also $L^{\alpha ^{\ast }(x,t)}\left( Q_{T} \right)
$ ($1<\left( \alpha ^{\ast }\right) ^{-}<\infty $) is a reflexive
space thus $\left\{ u_{m}\right\} _{m=1}^{\infty }$ has a
subsequence $\left\{ u_{m_{j}}\right\}
_{m=1}^{\infty }$ such that%
\begin{equation*}
a\left( x,t,u_{m_{j}}\right) \overset{\text{ }L^{\alpha ^{\ast
}(x,t)}\left( Q_{T}\right) }{\rightharpoonup }\psi.
\end{equation*}%
By the compact embedding \eqref{3.6}, we have,
\begin{equation*}
\exists \left\{ u_{m_{j_{k}}}\right\} _{m=1}^{\infty }\subset \left\{
u_{m_{j}}\right\} _{m=1}^{\infty }\text{, }u_{m_{j_{k}}}\overset{%
L^{p_{0}}\left( Q_{T}\right) }{\rightarrow }u_{0}\text{ }
\end{equation*}%
thus%
\begin{equation*}
u_{m_{j_{k}}}\underset{a.e}{\overset{Q_{T}}{\rightarrow }}u_{0}\text{ }
\end{equation*}%
and using the continuity of $a\left( x,t,.\right) $ for almost $\left(
x,t\right) \in Q_{T}$, we get%
\begin{equation*}
a(x,t,u_{m_{j_{k}}})\underset{a.e}{\overset{Q_{T}}{\rightarrow }}a\left(
x,t,u_{0}\right)
\end{equation*}%
so, we arrive at $\psi =a\left( x,t,u_{0}\right) $ i.e. $f_{2}(u_{m_{j_{k}}})%
\overset{L^{q_{0}}\left( 0,T;W^{-1,q_{0}}\left( \Omega \right) \right)
+L^{\alpha ^{\ast }\left( x,t\right) }\left( Q_{T}\right) }{\rightharpoonup }%
f_{2}\left( u_{0}\right) $.

Now let $a_{1}\left( x,t,u\right) :=g\left( x,t\right) \left\Vert
u\right\Vert _{L^{p}\left( \Omega \right) }\left( t\right) .$
Using the fact \eqref{3.6} and $p\leq p_{0}$, we have
\begin{equation*}
g\left( x,t\right) \left\Vert u_{m_{j}}\right\Vert _{L^{p}\left(
\Omega \right) }^{s}\left( t\right) \overset{L^{q_{0}}\left(
0,T;W^{-1,q_{0}}\left( \Omega \right) \right)
}{\longrightarrow}g\left( x,t\right) \left\Vert u_{0}\right\Vert
_{L^{p}\left( \Omega \right) }^{s}\left( t\right) .
\end{equation*}%
Therefore $a_{1}$ is weakly compact from $S_{0}$ into $L^{q_{0}}\left(
0,T;W^{-1,q_{0}}\left( \Omega \right) \right) +L^{\alpha ^{\ast }\left(
x,t\right) }\left( Q_{T}\right) .$ As a conclusion, $f$ is weakly compact
from $S_{0}$ into $L^{q_{0}}\left( 0,T;W^{-1,q_{0}}\left( \Omega \right)
\right) +L^{\alpha ^{\ast }\left( x,t\right) }\left( Q_{T}\right) .$
\end{proof}

\medskip

Now we give the proof of main theorem of this section.\medskip

\noindent \textbf{Proof of Theorem \ref{var}}. Since $A=Id,$ so
obviously it is a linear bounded map and satisfies the conditions
(ii) of Theorem \ref{gex}. Furthermore for any $u\in
W_{0}^{1,p_{0}}\left( Q_{T}\right) $ the following
inequalities are satisfied:%
\begin{equation*}
\int\limits_{0}^{T}\left\langle u,u\right\rangle _{\Omega
}dt=\int\limits_{0}^{T}\left\Vert u\right\Vert _{L^{2}\left(
\Omega \right) }^{2}dt\geq M\left\Vert u\right\Vert
_{L^{q_{0}}\left( 0,T;W^{-1,q_{0}}\left( \Omega \right) \right)
}^{2}
\end{equation*}%
and%
\begin{equation*}
\int\limits_{0}^{t}\left\langle \frac{\partial u}{\partial \tau }%
,u\right\rangle _{\Omega }d\tau =\frac{1}{2}\left\Vert u\right\Vert
_{L^{2}\left( \Omega \right) }^{2}\left( t\right) \geq M\frac{1}{2}%
\left\Vert u\right\Vert _{W^{-1,q_{0}}\left( \Omega \right) }^{2}\left(
t\right) ,
\end{equation*}%
a.e. $t\in \left[ 0,T\right] $ ($M>0$ constant comes from
embedding inequality). Thus condition (iv) of Theorem \ref{gex} is
also satisfied. Consequently from Lemma \ref{lemcoercive}-Lemma
\ref{lemweak}, it follows that the mappings $f$ and $A$ satisfy
all the conditions of Theorem \ref{gex}. If we
apply Theorem \ref{gex} to problem \eqref{MP}, we obtain that problem %
\eqref{MP} is solvable in $S_{0}$ for any $h\in L^{q_{0}}\left(
0,T;W^{-1,q_{0}}\left( \Omega \right) \right) +L^{\alpha ^{\ast }\left(
x,t\right) }\left( Q_{T}\right) $ satisfying the following inequality%
\begin{align*}
&\sup \left \{ \frac{1}{\left[ u\right] _{L^{p_{0}}\left( 0,T;\mathring{S}%
_{1,\left( p_{0}-2\right) q_{0},q_{0}}\left( \Omega \right)
\right) }+\left\Vert u\right\Vert _{L^{\alpha \left( x,t\right)
}\left( Q_{T}\right) }}\int\limits_{0}^{T}\left\langle
h,u\right\rangle _{\Omega }dt: u\in Q_{0} \right\} <\infty
\end{align*}%
where $Q_{0}:=L^{p_{0}}\left( 0,T;W_{0}^{1,p_{0}}\left( \Omega
\right) \right) \cap L^{\alpha \left( x,t\right) }\left(
Q_{T}\right)$. If we consider the norm definition of $h$ in
$L^{q_{0}}\left( 0,T;W^{-1,q_{0}}\left( \Omega \right) \right)
+L^{\alpha ^{\ast }\left(
x,t\right) }\left( Q_{T}\right) $, we conclude that \eqref{MP} is solvable in $%
S_{0}$ for any $h\in L^{q_{0}}\left( 0,T;W^{-1,q_{0}}\left( \Omega
\right) \right)+L^{\alpha ^{\ast }\left( x,t\right) }\left(
Q_{T}\right) .$ In order to complete the proof, it remains to
remark that \eqref{MP} can be written in the form
\begin{align*}
\frac{\partial {u}}{\partial
{t}}=h\left(x,t\right)-F\left(x,t,u,D_{i}u\right),
\end{align*}
and under the conditions of Theorem \ref{var} right hand belongs
to $L^{q_{0}}\left( 0,T;W^{-1,q_{0}}\left( \Omega \right) \right)$
which implies $\partial u/\partial t\in L^{q_{0}}\left(
0,T;W^{-1,q_{0}}\left( \Omega \right) \right).$

\begin{remark}
We note that if $\alpha ^{+}<p_{0}$ for the function $\alpha
\left( x,t\right)$ in \eqref{esit1} then existence of the weak
solution for problem \eqref{MP} can be proved under more
general(weak) sufficient conditions.
\end{remark}
This is proved in the following theorem:
\begin{theorem}
Assume that \eqref{esit1} is hold; the inequalities $1\leq
s<p_{0}-1$ and $p\leq p_{0}$ are satisfied. If $1<\alpha ^{-}\leq
\alpha \left( x,t\right) \leq \alpha ^{+}<p_{0},$ $\left(
x,t\right) \in Q_{T}$ and $g\in L^{\frac{p_{0}}{p_{0}-\left( s+1\right) }%
}\left( 0,T;L^{\tilde{p_{0}}^{\ast }}\left( \Omega \right) \right)
,$ $a_{0}\in L^{\beta _{1}\left( x,t\right) }\left( Q_{T}\right)
$, $a_{1}\in L^{\alpha ^{\ast }\left( x,t\right) }\left(
Q_{T}\right) $ where $\beta _{1}\left( x,t\right)
:=\frac{p_{0}\alpha ^{\ast }\left( x,t\right) }{p_{0}-\alpha
\left( x,t\right) }$ then $\forall h\in L^{q_{0}}\left(
0,T;W^{-1,q_{0}}\left( \Omega \right) \right) $ problem \eqref{MP}
has a generalized solution in the space $L^{p_{0}}\left(
0,T;\mathring{S}_{1,\left( p_{0}-2\right) q_{0},q_{0}}\left(
\Omega \right) \right) \cap W^{1,q_{0}}\left(
0,T;W^{-1,q_{0}}\left( \Omega \right) \right) $.
\end{theorem}
\begin{proof}
Using \eqref{esit1} we have
\begin{align*}
\left\langle f\left( u\right) ,u\right\rangle _{Q_{T}}&\geq
\sum_{i=1}^{n}\left( \int\limits_{0}^{T}\int\limits_{\Omega
}\left\vert u\right\vert ^{p_{0}-2}\left\vert D_{i}u\right\vert
^{2}dxdt\right) -\int\limits_{Q_{T}}\left\vert a_{0}\left(
x,t\right) \right\vert \left\vert
u\right\vert ^{\alpha \left( x,t\right) }dxdt \\
& -\int\limits_{Q_{T}}\left\vert a_{1}\left( x,t\right)
\right\vert dxdt-\int\limits_{0}^{T}\int\limits_{\Omega
}\left\vert g\left( x,t\right) \right\vert \left\Vert u\right\Vert
_{L^{p}\left( \Omega \right) }^{s}\left\vert u\right\vert dxdt.
\end{align*}%
For arbitrary $\epsilon >0$ estimating the second integral above
by Young's inequality and using $L^{p_{0}}\left(
0,T;\mathring{S}_{1,\left( p_{0}-2\right) q_{0},q_{0}}\left(
\Omega \right) \right) \subset L^{p_{0}}\left( Q_{T}\right)
$(Theorem \ref{embed}) we obtain the following inequality which
gives the coercivity of $f$,
\begin{equation*}
\left\langle f\left( u\right) ,u\right\rangle _{Q_{T}}\geq C_{5}\left[ u%
\right] _{L^{p_{0}}\left( 0,T;\mathring{S}_{1,\left( p_{0}-2\right)
q_{0},q_{0}}\left( \Omega \right) \right) }^{p_{0}}-\tilde{K}.
\end{equation*}%
here $C_{5}=C_{5}\left( p_{0},\left\vert \Omega \right\vert ,s\right) $ and \\ $%
\tilde{K}=\tilde{K}\left( \epsilon ,\left\Vert a_{0}\right\Vert _{L^{\beta
_{1}\left( x,t\right) }\left( Q_{T}\right) },\left\Vert a_{1}\right\Vert
_{L^{\alpha ^{\ast }\left( x,t\right) }\left( Q_{T}\right) },\left\Vert
g\right\Vert _{L^{\frac{p_{0}}{p_{0}-\left( s+1\right) }}\left( 0,T;L^{%
\tilde{p_{0}}^{\ast }}\left( \Omega \right) \right) }\right).$
\par Using the embedding
\begin{equation*}
L^{p_{0}}\left( 0,T;\mathring{S}_{1,\left( p_{0}-2\right)
q_{0},q_{0}}\left( \Omega \right) \right) \subset L^{p_{0}}\left(
Q_{T}\right)\subset L^{\alpha \left( x,t\right)}\left(
Q_{T}\right),
\end{equation*}
weak compactness and boundedness of \\$f:L^{p_{0}}\left(
0,T;\mathring{S}_{1,\left( p_{0}-2\right) q_{0},q_{0}}\left(
\Omega \right) \right) \cap W^{1,q_{0}}\left(
0,T;W^{-1,q_{0}}\left( \Omega \right) \right)\rightarrow
L^{q_{0}}\left( 0,T;W^{-1,q_{0}}\left( \Omega \right) \right)$
follows from Lemma \ref{lembounded} and \ref{lemweak}. Thus by the
virtue of the proof of Theorem \ref{var}, the proof is completed.
\end{proof}

\section{Homogeneous Case}
In this section, we investigate problem \eqref{MP} in homogeneous
case. We will obtain sufficient conditions which ensure that
problem \eqref{MP} has only trivial solution under these
conditions.
\begin{theorem}\label{beh}Let conditions of Theorem \ref{var} be fulfilled with the following
assumptions:
\begin{itemize}
\item [(i)] Let $h(x,t)=0$ and $p=2$, $p_{0}>2.$
\item[(ii)] Condition \eqref{esit2} is satisfied with
$a_{3}(x,t)=0.$
\item[(iii)]The functional $\Vert g
\Vert_{L^{2}(\Omega)}(t)$ is bounded for almost every
$t\in\mathbb{R^{+}}$,
\end{itemize}
then problem \eqref{MP} has only trivial solution.
\end{theorem}
\begin{proof}
Conditions of Theorem \ref{beh} provide that problem \eqref{MP}
has a solution in $S_{0}$. It follows from Definition
\ref{weakdef} that every weak solution satisfies the following
relation,
\begin{equation*}
 \frac{1}{2}\frac{d}{dt}\Vert u\Vert_{L^{2}(\Omega) }^{2}+\sum_{i=1}^{n}\int\limits_{\Omega }\left(
\left\vert u\right\vert ^{p_{0}-2}(D_{i}u)^{2}\right)dx
+\int\limits_{\Omega }a\left( x,t,u\right) udx+\int\limits_{\Omega
}g\left( x,t\right) \left\Vert u\right\Vert _{L^{2}\left( \Omega
\right) }^{s}udx=0
\end{equation*}
then we get
\begin{equation*}
 \frac{1}{2}\frac{d}{dt}\Vert u\Vert_{L^{2}(\Omega) }^{2}+\frac{4}{p_{0}^{2}}\sum_{i=1}^{n}\int\limits_{\Omega }
(D_{i}(\left\vert u\right\vert^{\frac{p_{0}}{2}}))^{2}dx
+\int\limits_{\Omega }a\left( x,t,u\right) udx+\int\limits_{
\Omega }g\left( x,t\right) \left\Vert u\right\Vert _{L^{2}\left(
\Omega \right) }^{s}udx=0
\end{equation*}
by using imbedding inequality and condition (ii), we obtain that
\begin{equation*}
 \frac{1}{2}\frac{d}{dt}\Vert u\Vert_{L^{2}(\Omega) }^{2}+\frac{4}{p_{0}^{2}c}\int\limits_{\Omega }
\left\vert u\right\vert^{p_{0}}dx+\int\limits_{\Omega }g( x,t)
\Vert u\Vert _{L^{2}\left( \Omega \right) }^{s}udx\leq 0
\label{beh-1}
\end{equation*}
by using Hölder inequality and condition (iii) for the last term
we have
\begin{equation*}
 \frac{1}{2}\frac{d}{dt}\Vert u\Vert_{L^{2}(\Omega) }^{2}+\frac{4}{p_{0}^{2}c}\int\limits_{\Omega }
\left\vert u\right\vert^{p_{0}}dx-K \Vert u\Vert _{L^{2}\left(
\Omega \right) }^{s+1}\leq 0 \label{beh-1}
\end{equation*}
where $K>0$ is a constant and by using embedding inequality,
\begin{equation*}
 \frac{1}{2}\frac{d}{dt}\Vert u\Vert_{L^{2}(\Omega) }^{2}+\frac{4}{p_{0}^{2}c(\emph{\emph{meas}}(\Omega))^{\frac{p_{0}-2}{2}}}\|u\|_{L_{2}(\Omega)}^{p_{0}}-K \Vert u\Vert _{L^{2}\left( \Omega \right)
}^{s+1}\leq 0 \label{beh-1}
\end{equation*}
denoting by $y=\|u\|_{L_{2}(\Omega)}^{2}$ and
$\mu=\frac{p_{0}}{2}$, we have
\begin{equation*}
 \frac{1}{2}\frac{dy}{dt}+\frac{4}{p_{0}^{2}c(\emph{\emph{meas}}(\Omega))^{\frac{p_{0}-2}{2}}}y^{\mu}-K y^{\frac{s+1}{2}}\leq 0 \label{beh-1}
\end{equation*}
by using Young inequality for the last term we have
\begin{equation*}
 \frac{1}{2}\frac{dy}{dt}+(\frac{4}{p_{0}^{2}c(\emph{\emph{meas}}(\Omega))^{\frac{p_{0}-2}{2}}}-K\varepsilon)y^{\mu}-Kc(\varepsilon) y\leq 0 \label{beh-1}
\end{equation*}
where $\varepsilon<
\frac{4}{Kp_{0}^{2}c\emph{\emph{meas}}(\Omega)^{\frac{p_{0}-2}{2}}}$
then we have,
\begin{equation*}
 \frac{1}{2}\frac{dy}{dt}\leq Kc(\varepsilon) y \label{beh-1}
\end{equation*}
by solving the last inequality and considering $y(0)=0$ we arrive
at the desired result.
\end{proof}

\end{document}